\documentclass[12pt]{amsart}

\usepackage{etex}
\usepackage[english]{babel}
\usepackage[cp1250]{inputenc}
\usepackage{amssymb,amsthm,amsfonts,amsmath}
\usepackage{mathrsfs}
\usepackage{verbatim}
\usepackage{ulem}
\usepackage{mathtools}

\usepackage{url}
\usepackage{amsopn}
\usepackage{graphicx}
\usepackage[usenames, dvipsnames]{xcolor}
\DeclareGraphicsExtensions{.pdf,.png,.jpg}

\usepackage[shortlabels]{enumitem}
\usepackage{titletoc}
\definecolor{darkblue}{rgb}{0.0, 0.0, 0.55}
\usepackage[colorlinks,linkcolor=BrickRed,citecolor=OliveGreen,urlcolor=darkblue,hypertexnames=true,backref=true]{hyperref}
\usepackage[a4paper,margin=2.5cm]{geometry}
\DeclareMathOperator{\mconv}{mconv}

\numberwithin{equation}{section}

\newcommand{\C}{\mathbb C}

\newcommand{\RR}{\mathbb R}

\newcommand{\CC}{\mathbb C}

\newcommand{\cD}{\mathcal D}

\newcommand{\cP}{\mathcal P}

\newcommand{\cS}{\mathcal S}
\newcommand{\cT}{\mathcal T}

\newcommand{\cW}{\mathcal W}

\newcommand{\free}{\partial^{\mathrm{free}}}
\newcommand{\matex}{\partial^{\mathrm{mat}}}
\newcommand{\Euc}{\partial^{\mathrm{Euc}}}

\newcommand{\eric}{\color{orange}}

\newtheorem{theorem}{Theorem}[section]
\newtheorem{corollary}[theorem]{Corollary}
\newtheorem{lemma}[theorem]{Lemma}
\newtheorem{proposition}[theorem]{Proposition}
\newtheorem{question}{Question}

\theoremstyle{definition}
\newtheorem{definition}[theorem]{Definition}
\newtheorem{remark}[theorem]{Remark}
\newtheorem{example}[theorem]{Example}

\title[MCS over the Euclidean ball and polar duals of real free spectrahedra]{Matrix convex sets over the Euclidean ball and polar duals of real free spectrahedra}

\author{Eric Evert}
\author{Benjamin Passer}

\newcommand{\Addresses}{{
    \bigskip
    \footnotesize

    \noindent\textsc{Department of Mathematics, University of Florida, Gainesville, FL, United States}\par\nopagebreak \textit{E-mail address}: \texttt{ericevert@ufl.edu}

    \medskip \medskip
    
    \noindent \textsc{Department of Mathematics, United States Naval Academy, Annapolis, MD, United States}\par\nopagebreak \textit{E-mail address}: \texttt{passer@usna.edu}
 }}

\begin{document}

\begin{abstract}
We show that the free spectrahedron determined by universal anticommuting self-adjoint unitaries is not equal to the minimal matrix convex set over the ball in dimension three or higher. This example, as well as other matrix convex sets over the ball, then provides context for structure results on the extreme points of coordinate projections. In particular, we show that the free polar dual of a real free spectrahedron is rarely the projection of a real free spectrahedron, contrasting a prior result of Helton, Klep, and McCullough over the complexes. We use this to show that spanning results for free spectrahedra that are closed under complex conjugation do not extend to free spectrahedrops that meet the same assumption. These results further clarify the role of the coefficient field.
\end{abstract}

\maketitle

\vspace*{-.19 in} 


\vspace*{.1 in} 


\vspace*{-.08 in} 

\section{Introduction}

Dilation theory provides a framework through which a matrix or operator may be replaced with a more straightforward counterpart, without sacrificing any key structure of the original. A key example is the demonstration of von Neumann's inequality for contractions through dilation to unitaries; see \cite[\S 1]{Sh21} for an introduction to these ideas. In the study of matrix convexity, dilations play a key role in finding special classes of extreme points \cite{kls14}, as in the dual setting of operator systems \cite{Arv08, DK15}. These extreme points are studied especially in the setting of free spectrahedra, where a restriction of finite-dimensionality for the extreme points is especially relevant \cite{EH, EHKM}. 

Spectrahedra have a much richer structure than their counterparts, polyhedra. For instance, while the closed Euclidean disk is certainly not a polyhedron, the linear matrix inequality  
\[
\begin{bmatrix} 1 & 0 \\ 0 & -1 \end{bmatrix} x_1 \, + \, \begin{bmatrix} 0 & 1 \\ 1 & 0 \end{bmatrix} x_2 \,\, = \,\, \begin{bmatrix} x_1 & x_2 \\ x_2 & -x_1 \end{bmatrix} \, \preceq \, I
\]
determines the spectrahedron $\overline{\mathbb{D}}$. Such a linear matrix inequality naturally extends to the dimension-free setting by allowing matrices to enter the inequality through the Kronecker product. That is, one obtains a free spectrahedron by considering those self-adjoint matrix pairs $(X_1,X_2)$ such that 
\[ \begin{bmatrix} 1 & 0 \\ 0 & -1 \end{bmatrix} \otimes X_1 \, + \, \begin{bmatrix} 0 & 1 \\ 1 & 0 \end{bmatrix} \otimes X_2 \,\, = \,\, \begin{bmatrix} X_1 & X_2 \\ X_2 & -X_1 \end{bmatrix} \, \preceq \, I.\]
 The resulting free spectrahedron is a matrix convex set whose first level is the disk. This inequality does not define the only matrix convex set over the closed disk; however, it notably defines the \textit{minimal} such example by \cite[Proposition 14.14]{HKMS19}.

It is common to understand convex sets through their extreme points. However, matrix convex sets admit several distinct notions of extreme point \cite{EHKM}, with different spanning properties; see \cite{EPS_expo} for a summary. A defining feature of minimal matrix convex sets is that many of these notions coincide and, moreover, only occur at the first matricial level. Thus, showing the existence of extreme points in the dimension-free sense at higher matricial levels is a method for certifying that a matrix convex set is not minimal. In Theorem \ref{thm:RealSpinBallFreeExtreme}, we use this approach to show that the natural extension of the above linear matrix inequality to three or more variables does not determine the minimal matrix convex set over the closed Euclidean ball.

Another key tool in convexity is the polar dual. This tool extends to the free polar dual in the dimension-free setting and has been particularly useful in the study of free spectrahedra. For example, \cite[Theorem 4.11]{HKM17} shows that, over the complexes, the class of projections of free spectrahedra is closed under free polar duality. That is, the free polar dual of the projection of a complex free spectrahedron is the projection of a complex free spectrahedron. We show in Theorem \ref{cor:drop_polyhedron_bounds} that this result does not hold in the real setting. More precisely, generalizing \cite[Corollary 6.3]{EHKM}, we show that if the free polar dual of a bounded real free spectrahedron is the projection of a real free spectrahedron, then the first level must be a polyhedron. We later use this fact to prove the existence of closed bounded free spectrahedrops that are closed under complex conjugation but have no free extreme points. 

The remainder of the introduction showcases our notation and some key definitions. Section \ref{sec:ball} examines matrix convex sets over the Euclidean ball. In particular, Theorem \ref{thm:RealSpinBallFreeExtreme} describes a change in behavior between $2$ variables and $3$ or more variables for the smallest such matrix convex set. We then prove results about spectrahedrops in section \ref{sec:drops}, using our previous examples to clarify boundary cases. Corollary \ref{corollary:SpectrahedropWithNoFreeExtreme} shows that free spectrahedrops that are closed under complex conjugation need not have any free extreme points, in stark contrast to free spectrahedra. Such examples further clarify the role of the coefficient field, especially as one extends a free spectrahedrop into a free spectrahedron in more variables. Along the way, we pose multiple open questions that are of interest.

\subsection{Preliminaries and notation}

\begin{definition}
Let $\mathbb{F} \in \{ \mathbb{R}, \mathbb{C}\}$, and let $A = (A_1, \ldots, A_g) \in SM_n(\mathbb{F})^g$, where $SM_n(\mathbb{F})^g$ denotes the collection of $g$-tuples of $n \times n$ self-adjoint matrices with entries in $\mathbb{F}$. Then the monic linear pencil
\[ L_A(X) \, := \, I - A_1 \otimes X_1 - \ldots - A_g \otimes X_g \]
determines the \textbf{free spectrahedron}
\[ \cD_A^{\mathbb{F}} \,\, = \,\, \bigcup\limits_{n=1}^\infty \cD_A^{\mathbb{F}}(n) \,\, = \,\, \bigcup\limits_{n=1}^\infty \{ X \in SM_n(\mathbb{F})^g: L_A(x) \succeq 0 \}.\]
If the coefficient field is omitted, we assume $\mathbb{C}$. We will often make use of the strictly linear part of the pencil, which is denoted
\[
\Lambda_A(X) \, := \, A_1 \otimes X_1 + \ldots + A_g \otimes X_g.
\]
\end{definition}

A complex free spectrahedron need not be closed under entrywise complex conjugation. However, any complex free spectrahedron that is closed under complex conjugation admits a real coefficient tuple $A$ that represents it, in such a way that the structure of $\cD_A^{\mathbb{R}}$ and $\cD_A^{\mathbb{C}}$ are closely tied (see \cite[Lemma 3.3]{EH} and \cite[Theorem 2.4]{Eve23}), and both sets admit clean Krein-Milman theorems for free extreme points by \cite[Theorem 1.3]{EH}. See Definition \ref{def:extreme_points_summary} below for the definitions of these terms. However, the structure of complex free spectrahedra that are not closed under complex conjugation is quite different, as in \cite[Example 6.30]{Kr} and \cite[\S 2]{Pas22}. Hence, we typically work in complex coefficients and clarify as necessary.

Free spectrahedra are important examples of matrix convex sets. For simplicity, all references to matrix convexity in what follows will refer to matrix convex sets over Euclidean space, as in the following definition.

\begin{definition}
Fix $d \geq 1$ and let $K = \bigcup\limits_{n=1}^\infty K_n \subseteq \bigcup\limits_{n=1}^\infty M_n^d$. Then $K$ is called \textbf{matrix convex} if $K$ is closed under direct sums and under images of unital completely positive (UCP) maps. Equivalently, $K$ is closed under direct sums and compressions. Equivalently, $K$ is closed under \textbf{matrix convex combinations} of $X^{(1)}, \ldots, X^{(m)} \in K$, which are sums $\sum\limits_{i=1}^m  V_i^* X^{(i)} V_i$  where the $V_i$ are $n_i \times n$ matrices with $\sum\limits_{i=1}^n V_i^* V_i = I_n.$
\end{definition}

While we will not need the noncommutative convexity theory of \cite{DK}, which allows for infinite levels, we colloquially refer to \lq\lq level infinity\rq\rq\hspace{0pt} of a matrix convex set whenever this simplifies matters. See also \cite{BM_mcs, BR_op} for a related study using the real coefficient field. Note that the above definition does not require $K$ to consist of self-adjoints, and it is natural to break a non self-adjoint matrix convex set into self-adjoint components in the typical way. This is usually a matter of convenience, but we caution the reader as in \cite[Remark 2.2]{Pas22} that the discussion of the coefficient field of a free spectrahedron is always done in the self-adjoint setting. Similarly, regardless of the coefficient field, our discussions of the free polar dual will be in the self-adjoint setting only.

Matrix convex sets over Euclidean space are dual to finite-dimensional operator systems (see \cite[Chapter 13]{Paulsen_book}), and one way to package this information is in the matrix range of a tuple of operators.

\begin{definition}
Let $T = (T_1, \ldots, T_d) \in B(H)^d$. Then the {\textbf{matrix range}} of $T$ is the levelwise closed and bounded matrix convex set defined by
\[ \cW(T) \,\, = \,\, \bigcup\limits_{n=1}^\infty \cW_n(T) \,\, = \,\, \bigcup\limits_{n=1}^\infty \{\phi(T): \phi: B(H) \to M_n \text{ is UCP}\}.\]
\end{definition}

In fact, every closed and bounded matrix convex set may be expressed as $\cW(T)$ for some tuple of operators. Uniqueness properties of this tuple are typically only found under rigid conditions concerning compactness \cite{DDOSS17, PS19, DP22}, and these results are related to earlier uniqueness results in free spectrahedra \cite{EHKM, HKM13, Zal17}. When $A$ is a tuple of matrices, $\cW(A)$ is in fact the matrix convex hull of $A$, which we denote by $\text{mconv}(A)$. Moreover, if $A$ is appropriately positioned, $\cW(A)$ and $\cD_A$ are related by the free polar dual, as we will describe below.

In this manuscript, we consider the self-adjoint free polar dual only, for simplicity. We also simplify notation in what follows by letting 
\[ SM^g \,\, := \,\, \bigcup\limits_{n=1}^\infty SM_n^g \]
denote the collection of $g$-tuples of self-adjoint matrices. As with previous notation choices, the complex field is assumed if one is not specified.

\begin{definition}
Let $K$ be a matrix convex set in $g$ self-adjoint variables. Then $K^\circ$, the \textbf{free polar dual} of $K$, is defined as the matrix convex set in $g$ self-adjoint variables given by
\[ K^\circ \,\, := \,\, \left\{X \in SM^g: \,\ \forall \, Y \in K, \,\, \sum\limits_{i=1}^g Y_i \otimes X_i \, \preceq \, I \right\}.\]
\end{definition}

If $A$ is a tuple of self-adjoint matrices such that $0 \in \cW_1(A)$, then \cite[Proposition 3.1 and Lemma 3.2]{DDOSS17} show that
\[ \cD_A^\circ = \cW(A) \hspace{.5 in} \text{ and } \hspace{.5 in} \cW(A)^\circ = \cD_A,\]
where again we recall that $\cW(A) = \text{mconv}(A)$ for matrix tuples. As such, there is a connection between uniqueness theorems for matrix ranges and uniqueness theorems for free spectrahedra, as referenced above.

\begin{definition}
Let $K$ be a matrix convex set in $h$ self-adjoint variables, and let $g < h$. Then $\mathcal{P}_g K$ denotes the \textbf{coordinate projection} of $K$ onto the first $g$ coordinates. That is, $X \in \cP_g K$ if and only if there exists $Y$ (a length $h-g$ tuple of self-adjoint matrices the same dimension as the $X_i$) such that $(X, Y) \in K$. The projection of a free spectrahedron is called a \textbf{free spectrahedrop}. 
\end{definition}

The study of free spectrahedrops is more delicate than that of free spectrahedra. While free spectrahedra are the only free semialgebraic matrix convex sets by \cite[Theorem 1.4]{HM12}, not every free spectrahedrop is a free spectrahedron by \cite[Proposition 9.8]{HM12}. Moreover, the free polar dual of a complex free spectrahedrop with nonempty interior is also a complex free spectrahedrop by \cite[Corollary 4.17]{HKM17}, which implies that any matrix convex hull of a matrix tuple, so long as it includes the point $0$, is also a complex free spectrahedrop. 

The above is in stark contrast with free spectrahedra themselves; very few examples of free spectrahedra whose free polar duals are free spectrahedra are known. See, e.g., \cite[\S 6]{EHKM}. In this article, we present a new class of complex free spectrahedra whose free polar duals are free spectrahedra in Theorem \ref{theorem:DualFreeSpecs}. In a special case, the complex free spectrahedron is in fact self-dual; see Theorem \ref{thm:Paulidual}. In contrast, as a consequence of \cite[Proposition 6.1 and Corollary 6.3]{EHKM} a real free spectrahedron can never be self-dual. See Remark \ref{rem:nobigdual} for further discussion.

As in classical convexity, matrix convex sets and their counterparts are studied through their extreme points. However, in this setting there are multiple types of extreme point to consider, with various spanning properties.

\begin{definition}\label{def:extreme_points_summary}
Let $K$ be a matrix convex set, and let $X \in K_n$.
\begin{enumerate}
    \item $X$ is called a \textbf{Euclidean extreme point} or \textbf{classical extreme point} of $K$ (denoted $X \in \Euc K$) if $X$ is an extreme point of the convex set $K_n$.
    \item $X$ is called a \textbf{matrix extreme point} of $K$ (denoted $X \in \matex K$) if whenever $K$ is expressed as a matrix convex combination of points $X^{(i)} \in K$, with each $V_i \not= 0$ and each $X^{(i)}$ of matrix dimension at most $n$, then each $X^{(i)}$ is unitarily equivalent to $X$.
    \item $X$ is called a \textbf{free extreme point} of $K$ (denoted $X \in \free K$) if whenever $K$ is expressed as a matrix convex combination of points $X^{(i)} \in K$, with each $V_i \not= 0$, then each $X^{(i)}$ is either unitarily equivalent to $X$ or is unitarily equivalent to a reducible matrix tuple with $X$ as a summand.
\end{enumerate}
\end{definition}

 In the case of bounded complex conjugation closed free spectrahedra, free extreme points are of particular note, as they form the minimal set that spans the free spectrahedron through matrix convex combinations (see \cite[Theorem 1.3]{EH}). We will show in Corollary \ref{corollary:SpectrahedropWithNoFreeExtreme} that this spanning result for free extreme points of complex conjugation closed free spectrahedra does not extend to complex conjugation closed free spectrahedrops. 

General closed bounded matrix convex sets are the closed matrix convex hull of their matrix extreme points by \cite[Corollary 3.6]{WW}, and in fact matrix extreme points correspond to finite-dimensional pure UCP maps by \cite[Theorem B]{FarenickExtreme}. Free extreme points (formerly called \textbf{absolute extreme points}) correspond to finite-dimensional boundary representations by \cite[Theorem 1.2]{kls14}. Thus, in analogy with \cite{Arv08}, free extreme points are the maximal irreducible elements of $K$. That is, $X = (X_i)_{i=1}^g \in K$ is free extreme if and only if $X$ is irreducible and the only dilations $Y_i = \begin{bmatrix} X_i & A_i \\ B_i & C_i \end{bmatrix}$ with $Y = (Y_i)_{i=1}^g \in K$ are the trivial dilations (each $A_i$ and $B_i$ is zero). 

A maximal element that is not necessarily irreducible is called an \textbf{Arveson extreme point}, or equivalently an Arveson extreme point is a direct sum of finitely many free extreme point(s). This relaxation is useful when determining points of $K$ using an algebraic identity, where the insistence on irreducibility is artificial. It also allows for dilation results to be stated more cleanly, as every matrix convex combination may be expressed as the compression of a direct sum. For instance,
\begin{equation}\label{eq:arv_equivalence} K = \mconv(\free K) \,\,\, \iff \,\,\, \text{every } X \in K \text{ dilates to an Arveson extreme point}. \end{equation}
This point of view is taken, for example, in \cite{EH}.

We caution the reader that the matrix convex hull in (\ref{eq:arv_equivalence}) does \textit{not} include a closure operation, and that the above extreme points are finite-dimensional in nature (unlike general boundary representations). For free extreme points, the assumption of finite-dimensionality makes the associated Krein-Milman or Caratheodory problems more restrictive.

Our results in Section \ref{sec:ball} focus on using free extreme points to distinguish matrix convex sets over the ball, including examples that come naturally from a non self-adjoint presentation. In Section \ref{sec:drops}, we focus on the extreme points of coordinate projections, using our study of the ball to generate boundary cases for structure theorems. We in particular show that the free polar dual of a complex conjugation closed free spectrahedron is rarely the projection of a complex conjugation closed free spectrahedron. We then use this result to construct free spectrahedrops that are closed under complex conjugation but have no free extreme points, illustrating that the spanning result of \cite{EH} for complex conjugation closed free spectrahedra does not extend to free spectrahedrops.

\section{Anticommutation}\label{sec:ball}

Anticommuting matrices are key in the study of spin systems, as in \cite{AP03} and \cite{FarenickSpinOriginal, FarenickSpinCorrected}, and they have also been used as boundary cases for dilation problems, as in \cite{DDOSS17} and \cite{PasShaSol18}. The pair $\left( \begin{bmatrix} 1 & 0 \\ 0 & -1 \end{bmatrix}, \begin{bmatrix} 0 & 1 \\ 1 & 0 \end{bmatrix} \right)$ not only contains self-adjoint unitary matrices that anticommute with each other, but the pair is universal among all pairs of operators meeting the same relations. This may be seen by noting that the $C^*$-algebra generated by the pair has the maximum possible vector space dimension that the relations allow, namely $2^2 = 4$. 

This idea extends to larger length tuples by a well-known recursive procedure. If $(F_1, \ldots, F_g)$ are universal, pairwise anticommuting, self-adjoint unitaries, then we consider the tuple
\[ \left( F_1 \otimes \begin{bmatrix} 1 & 0 \\ 0 & -1 \end{bmatrix}, F_2 \otimes \begin{bmatrix} 1 & 0 \\ 0 & -1 \end{bmatrix}, \ldots, F_g \otimes \begin{bmatrix} 1 & 0 \\ 0 & -1 \end{bmatrix}, I \otimes \begin{bmatrix} 0 & 1 \\ 1 & 0 \end{bmatrix} \right)\]
to achieve the same universal property in $g+1$ variables.

\begin{definition}
Fix $g \geq 2$. Let 
\[ F^{[g]} = (F_1, F_2, \ldots, F_g) \]
denote the $g$-tuple of universal, pairwise anticommuting, self-adjoint unitary operators found using the above tensor product procedure. When the tuple length is understood, we denote this tuple by $F$.
\end{definition}

For $g = 2$, the free spectrahedron $\cD_{F^{[2]}}$ is equal to the minimal matrix convex set over the closed disk by \cite[Proposition 14.14]{HKMS19}, so equivalently, its free polar dual $\cW(F^{[2]})$ is equal to the maximal matrix convex set over the closed disk. However, it should also be noted that $F_1 + iF_2$ is unitarily equivalent to $\begin{bmatrix} 0 & 2 \\ 0 & 0 \end{bmatrix} = 2E_{12}$, so these results are also corollaries of Ando's much earlier characterization of operators with numerical radius $1$ in \cite{Ando} (see \cite[Theorem 5.12]{FarenickSpinOriginal} for one description).

For $g \geq 3$, the tensor presentation of $F^{[g]}$ is not irreducible over $\mathbb{C}$. Specifically, $C^*(F^{[3]}) = M_2 \oplus M_2$ and 
\begin{equation} F^{[3]} \cong P \oplus \overline{P}, \end{equation}
where
\begin{equation}\label{eq:PauliMatrices} P = (P_1, P_2, P_3) = \left( \begin{bmatrix} 1 & 0 \\ 0 & -1 \end{bmatrix}, \begin{bmatrix} 0 & 1 \\ 1 & 0 \end{bmatrix}, \begin{bmatrix} 0 & i \\ -i & 0 \end{bmatrix} \right) \end{equation}
and $\overline{P}$ denotes the entrywise complex conjugate, meaning $\overline{P} = (P_1, P_2, -P_3)$. As a consequence, one has
\[
\cD_{F^{[3]}} = \cD_{P \oplus \overline{P}} = \cD_P \cap \cD_{\overline{P}}.
\]

 Whether or not the $g$-tuple $F$ for $g \geq 3$ has similar properties as for $g = 2$ is an open problem discussed in \cite[(4.12)]{Pas19} and similarly recast in the text directly after \cite[Theorem 5.12]{FarenickSpinOriginal}.

\begin{question}\label{ques:themain}
For $g \geq 3$, is $\cD_{F}$ equal to the smallest matrix convex set whose first level is the closed ball, denoted $\mathcal{W}^{\text{min}}(\overline{\mathbb{B}_d})$? Equivalently, is $\cW(F)$ equal to the largest matrix convex set over the closed ball, denoted $\mathcal{W}^{\text{max}}(\overline{\mathbb{B}_d})$?
\end{question}

We resolve this question negatively in Theorem \ref{thm:RealSpinBallFreeExtreme} and Corollary \ref{cor:spinwrapup}. The nature of the question is more clear with some additional structure of the Pauli matrices in plain view.

\begin{theorem}\label{thm:Paulidual}
The free spectrahedron $\cD_P$ satisfies $\cD_P = \cD_P^\circ$. Similarly, $\cD_{\overline{P}} = \cD_{\overline{P}}^\circ$. As a consequence, up to unitary equivalence, $P$ is the only free extreme point of $\cD_P$, and $\cD_P$ is the matrix convex hull of its free extreme points.
\end{theorem}
\begin{proof}
Consider the matrix range $\cW(P)$. Since $\cW(P)$ consists of the UCP images $X = \phi(P)$ where $\phi$ maps into a matrix algebra, we note that the operator system spanned by $P$ is $M_2$, so these UCP images are characterized by Choi's theorem \cite[Theorem 2]{Choi}. In particular, since 
\[ E_{11} = \cfrac{I_2 + P_1}{2}, \,\,\,\,\,\,\, E_{12} = \cfrac{P_2 - iP_3}{2}, \,\,\,\,\,\,\, E_{21} = \cfrac{P_2 + iP_3}{2}, \,\,\,\,\,\,\, E_{22} = \cfrac{I_2 - P_1}{2}, \]
a UCP image $X = \phi(P)$ is characterized by the condition
\[ \begin{bmatrix} \cfrac{I_2 + X_1}{2} & \cfrac{X_2 - iX_3}{2} \\ \cfrac{X_2 + iX_3}{2} & \cfrac{I_2 - X_1}{2} \end{bmatrix} \,\,\, \succeq \,\,\, 0,\]
which is equivalent (after scaling by $2$) to
\[ I \, + \, P_1 \otimes X_1 + P_2 \otimes X_2 - P_3 \otimes X_3 \,\,\, \succeq \,\,\, 0.\]
Because $P_3$ anticommutes with both $P_1$ and $P_2$, conjugating both sides by $P_3 \otimes I$ shows that this inequality is equivalent to
\[ I - P_1 \otimes X_1 - P_2 \otimes X_2 - P_3 \otimes X_3 \,\,\, \succeq \,\,\, 0.\]
That is, $X \in \cW(P)$ if and only if $X \in \cD_P$. A similar argument applies to the complex conjugate.

As previously discussed, the claim $\cD_P^\circ = \cW(P) = \mconv(P)$ follows from \cite[Proposition 3.1 and Lemma 3.2]{DDOSS17}, so we now have $\cD_P = \mconv(P)$. It is then straightforward to show that $P$ a free extreme point of $\cD_\cP$, and that every free extreme point of $\cD_P$ is unitarily equivalent to $P$. 
\end{proof}

Self-duality of a matrix convex set with respect to the self-adjoint free polar dual is not the same as the requirement that the corresponding operator system and its dual are completely order isomorphic. However, it is notable that the operator system $M_n$ is self-dual for all $n$ by \cite[Example 2.2]{Ng-P-16}, and when $n = 2$, \cite[Remark 3.5]{Ng-P-16} shows that the minimum norm product $\|\phi\|_{cb} \cdot \|\phi^{-1}\|_{cb} = 2$ is achieved for a complete order isomorphism $\phi$ between $M_2$ and its operator system dual. This provides context for Theorem \ref{thm:Paulidual}, as both $P$ and $\overline{P}$ span $M_2$ as an operator system. For $n > 2$, the same reference notes that $M_n$ and its dual operator system are still completely order isomorphic, but with maps that do not achieve the minimum possible product of norms. A similar phenomenon appears in our own results; the behavior witnessed in Theorem \ref{thm:Paulidual} does not extend to any tuple inside a higher matrix algebras $M_n$, where that tuple spans $M_n$ as an operator system (see Remark \ref{rem:nobigdual}).

Direct computation verifies that $\overline{P} \not\in \cD_P$ (in fact, this claim is equivalent to the fact that the transpose map on $2 \times 2$ matrices is not completely positive), $-P \not\in \cD_P$, and $(P_1, P_2, 0) \not\in \cD_P$, so the complex free spectrahedron $\cD_P$ is lacking in key symmetries that are present in $\cD_F = \cD_{P \oplus \overline{P}} = \cD_P \cap \cD_{\overline{P}}$. Notably, multiple symmetries found in the first level of $\cD_P$, which is $\overline{\mathbb{B}_3}$, are lacking in the matrix convex set $\cD_P$ as a whole. 

In fact, because $\cD_{F^{[3]}} = \cD_P \cap \cD_{\overline{P}} = \cW(P) \cap \cW(\overline{P})$ and the unital linear map sending $P \mapsto \overline{P}$ is the transpose map, elements of $\cD_{F^{[3]}}$ are images of $P$ under positive partial transpose (PPT) maps $\phi: M_2 \to M_n$ for various $n$. By definition, these are maps $\phi$ such that both $\phi$ and $\phi \circ T$ are UCP maps. The structure of PPT maps $\phi: M_2 \to M_n$ changes behavior at $n = 4$: $2 \times n$ PPT maps need not be separable for $n \geq 4$ by the results of \cite{Horo_1}. This strongly suggests that Question \ref{ques:themain} has a negative answer. However, an examination along these lines would not evidently give a clear description of free extreme points, which are crucial in the study of free spectrahedra. We give explicit examples of noncommuting free extreme points of $\cD_{F^{[3]}}$ in Theorem \ref{thm:RealSpinBallFreeExtreme}, of matrix dimension 4 and 6, as part of a negative resolution.

\begin{remark}
Theorem \ref{thm:Paulidual} should be contrasted with \cite[\S 9]{DDOSS17}, which studies the self-dual matrix ball,  
\begin{equation} \label{eq:selfdualDDSS} \mathfrak{D}_g \,\, = \,\, \left\{ X \in SM^g : \left\|\sum\limits_{i=1}^g X_i \otimes \overline{X_i} \, \right\| \, \leq 1  \right\}. \end{equation}
This set is characterized in \cite[Lemma 9.2 and Remark 9.3]{DDOSS17} as the unique matrix convex set that is closed under complex conjugation, has $\pm 1$ symmetry, and is self-dual with respect to the self-adjoint free polar dual. Theorem \ref{thm:Paulidual} shows that the symmetry or conjugation assumption is necessary. It should also be noted that $\mathfrak{D}_g$ is tied to the self-dual operator system $SOH$ of \cite{Ng-P-16}, as seen in the computations of \cite[Proposition 3.3]{Ng-P-16}.

Note in particular that $\mathfrak{D}_g$ cannot be a free spectrahedron. The fact that $\mathfrak{D}_g$ is closed under complex conjugation and is self-dual would show via \cite[Proposition 6.1 and Corollary 6.3]{EHKM} that if $\mathfrak{D}_g$ is a free spectrahedron, then its first level must be a polytope. However, the first level of $\mathfrak{D}_g$ is evidently the closed Euclidean ball. Said differently, the self-duality of the complex free spectrahedron $\cD_P$ in Theorem \ref{thm:Paulidual} underscores the fact that in other free spectrahedron structure results, the choice of the real coefficient field (or similarly an assumption of closedness under complex conjugation) is absolutely crucial.
\end{remark}

\begin{proposition}
$\cW(\overline{P}) = -\cW(P)$. Equivalently, $\cD_{\overline{P}} = -\cD_P$.
\end{proposition}
\begin{proof}
It suffices to prove the second claim. The second claim follows from applying a conjugation by $P_3 \otimes I$ to the definition of $\cD_{\overline{P}}$, where we note that $P_1 = \overline{P_1}$, $P_2 = \overline{P_2}$, $P_3 = -\overline{P_3}$, and $P_3 \otimes I$ anticommutes with $P_1 \otimes I$ and $P_2 \otimes I$.
\end{proof}

The above proposition again emphasizes the fact that a free spectrahedron or matrix range of anticommuting self-adjoint unitaries, such as $\cD_P$, need not enjoy very much symmetry. As such, the abundant symmetry  of $\cD_F$ stems from the \textit{universality} of the coefficients used. Another perspective is that if $A \in SM^g$ consists of anticommuting self-adjoint unitaries, then conjugations by $A_i \otimes I$ will demonstrate that $\cD_A$ has symmetry under negating \textit{all but one} of the generators, but symmetry under negating any \textit{one} of the generators is not even a given. On the other hand, consider any $g \times g$ orthogonal (real unitary) matrix $U$ and the tuple determined by multiplication $UX$, where $X = (X_1, \ldots, X_g)$ is viewed as a column vector. That is, consider
\begin{equation}\label{eq:ortho_example} U X = \left( \sum\limits_{k=1}^g u_{1k} X_k, \, \sum\limits_{k=1}^g u_{2k} X_k \, , \ldots, \, \sum\limits_{k=1}^g u_{gk} X_k \right),
\end{equation}
the orthogonal transformation of $X$ determined by $U$. 

\begin{proposition}\label{prop:symmetry}
Fix $g \geq 2$. Then for any $g \times g$ orthogonal $U$, $\cD_{UF} = \cD_F$ and $\cW(UF) = \cW(F)$.  
\end{proposition}
\begin{proof}
Let $U$ be an orthogonal matrix. Because $U$ has real entries, the tuple $UF$ consists of self-adjoints. The remaining relations, given by $F_i^2 = I$ and $F_i F_j = -F_j F_i$, are seen by direct computation using orthogonality: 
\[ \left( \sum\limits_{k=1}^g u_{ik} F_k \right) \left( \sum\limits_{m=1}^g u_{jm} F_m \right) \,\, = \,\, \sum\limits_{k=1}^g u_{ik} u_{jk} I \,\, + \,\, \sum\limits_{k \not= m} u_{ik} u_{jm} F_k F_m \]
\[ \,\, = \,\, \delta_{ij} I + \sum\limits_{k < m} (u_{ik}u_{jm} - u_{im}u_{jk}) F_k F_m.\]
If $i = j$, then only the delta term appears, with a result of $I$. If $i \not= j$, then only the sum appears, and since it is antisymmetric in $i$ and $j$, this demonstrates the anticommutation relation. This means that there is a unital $*$-homomorphism that maps $F \mapsto UF$, which is clearly invertible. Hence, the examination of UCP images shows that $\cW(F) = \cW(UF)$, which implies that $\cD_{F} = \cD_{UF}$.
\end{proof}

We can also apply the transformation to elements themselves, rather than the coefficient tuple.

\begin{corollary}\label{cor:symmetry2}
Fix $g \geq 2$. Then $\cD_F$ is closed under orthogonal transformations: if $X \in \cD_F$, then $UX \in \cD_F$. As a special case, if $(X_1, \ldots, X_g) \in \cD_F$, it follows that for any choice of $\pm$ signs, $(\pm X_1, \pm X_2, \ldots, \pm X_g) \in \cD_F$.
\end{corollary}
\begin{proof}
Grouping terms in a linear matrix inequality
\[ F_1 \otimes \left( \sum\limits_{k=1}^g u_{1k} X_k \right) \,\, + \,\, F_2 \otimes \left( \sum\limits_{k=1}^g u_{2k} X_k \right) \,\, + \,\, \ldots \,\, + \,\, F_g \otimes \left( \sum\limits_{k=1}^g u_{gk} X_k \right) \,\, \preceq \,\, I\] 
by $X_i$ instead of by $F_i$ shows that $UX \in D_F$ if and only if $X \in \cD_{U^T F}$. So, the desired symmetry reduces to the previous result. The claim about negating entries of $X$ follows by considering $U$ diagonal with $\pm 1$ entries.
\end{proof}

These symmetries make it very easy to extend results about $\cD_{F^{[g]}}$ to $\cD_{F^{[g+1]}}$.

\begin{corollary}\label{cor:extending_by_zero}
Fix $g < h$. Then $\cP_g \cD_{F^{[h]}} = \cD_{F^{[g]}}$, and in fact $X \in \cD_{F^{[g]}}$ if and only if $(X, 0) \in \cD_{F^{[h]}}$.
\end{corollary}
\begin{proof}
It suffices to consider $h = g + 1$. Fix $X \in \mathcal{P}_g \cD_{F^{[h]}}$ and a witness $(X, Y) \in \cD_{F^{[h]}}$. Since $(X, -Y)$ also belongs to $\cD_{F^{[h]}}$, averaging shows that $(X, 0) \in \cD_{F^{[h]}}$, and hence $(X, 0)$ may always be used as the witness. On the other hand, the construction of $F^{[h]}$ from $F^{[g]}$ uses the unitary tuple
\[ \left( \begin{bmatrix} F_1 & 0 \\ 0 & -F_1 \end{bmatrix}, \ldots,  \begin{bmatrix} F_g & 0 \\ 0 & -F_g \end{bmatrix}, \begin{bmatrix} 0 & I \\ I & 0 \end{bmatrix} \right), \]
so claiming $(X, 0) \in \cD_{F^{[h]}}$ is equivalent to claiming
\[ \begin{bmatrix} \sum_{i=1}^g F_i \otimes X_i & 0 \\ 0 & - \sum_{i=1}^g F_i \otimes X_i \end{bmatrix} \, \preceq \, I. \]
That is, both $\pm X \in \cD_{F^{[g]}}$, and by the symmetry of $\cD_{F^{[g]}}$ this is equivalent to the claim $X \in \cD_{F^{[g]}}$. All of these steps were equivalences, so $X \in \cD_{F^{[g]}}$ if and only if $(X, 0) \in \cD_{F^{[h]}}$, and consequently $\cP_g \cD_{F^{[h]}} = \cD_{F^{[g]}}$.
\end{proof}

However, the projections of $\cD_P$ and $\cD_{\overline{P}}$ behave very differently, and we will need these claims in the next section.

\begin{proposition}
\label{prop:max_proj_pauli}
The maximal matrix convex set over the closed disk is the projection of $\cD_P$ to the first two coordinates. That is, $\cP_2 \cD_P = \cW^{\operatorname{max}}(\overline{\mathbb{D}})$.  Similarly, $\cP_2 \cD_{\overline{P}} = \cW^{\operatorname{max}}(\overline{\mathbb{D}})$.
\end{proposition}

\begin{proof}
Because $P_1$, $P_2$, and $P_3$ are anticommuting self-adjoint unitaries, the $C^*$-norm identity shows that $\cW_1(P)$ is contained in the closed unit ball. So, the coordinate projection has $\cP_2 \cW_1(P) \subseteq \overline{\mathbb{D}}$, which implies $\cP_2 \cW(P) \subseteq \cW^{\text{max}}(\overline{\mathbb{D}})$. On the other hand, 
\[ \cW^{\text{max}}(\overline{\mathbb{D}}) \,\, = \,\, \cW\left( F^{[2]} \right) \,\, = \,\, \cW(P_1, P_2) \,\, \subseteq \,\, \cP_2 \cW(P), \]
so the reverse containment also holds. A similar argument applies to $\overline{P}$.
\end{proof}

\begin{theorem}
\label{thm:RealSpinBallFreeExtreme}
Let $n = 4$ or $n = 6$. Then $(\free \cD_{F^{[3]}})(n)$ is nonempty. As an immediate consequence, $\cD_{F^{[3]}}$ is not a minimal matrix convex set, and its free polar dual $\cW(F^{[3]})$ is not a maximal matrix convex set. 

More precisely, for $n=4$ or $n=6$ there exists a symmetric matrix tuple $C=(C_1,C_2,C_3) \in M_{n/2} (\CC)$ such that $X$ defined by
\begin{equation}
\label{eq:FreeExForm}
X = 
    \left( 
        \begin{bmatrix} 0 & C_1 \\
        \overline{C}_1 & 0 \end{bmatrix},
        \begin{bmatrix} 0 & C_2 \\
        \overline{C}_2 & 0 \end{bmatrix},
        \begin{bmatrix} 0 & C_2 \\
        \overline{C}_2 & 0 \end{bmatrix}
    \right) 
    \in \cD_{F^{[3]}}(n) 
\end{equation}
is a free extreme point at level $n$ of $\cD_{F^{[3]}}$.
\end{theorem}
\begin{proof}
    As a minimal matrix convex set is the matrix convex hull of its first level, all free extreme points of a minimal matrix convex set must be at level $1$. Therefore, if $\cD_{F^{[3]}}$ has free extreme points at levels other than $1$, it is immediate that $\cD_{F^{[3]}}$ is not minimal. Furthermore, a bounded matrix convex set containing $0$ is minimal if and only if its free polar dual is maximal, so it will follow that the free polar dual of $\cD_{F^{[3]}}$ is not maximal. 

    Now, in the case that $n=4$, let 
    \begin{equation}
    \label{eq:FreeExForm4}
    C = \frac{1}{2}\left(\left[
\begin{array}{cc}
 1+\frac{1}{\sqrt{3}} & 0 \\
 0 & -1+\sqrt{3} \\
\end{array}
\right],\left[
\begin{array}{cc}
 0 & 1 \\
 1 & 0 \\
\end{array}
\right],\left[
\begin{array}{cc}
 -i+\frac{2 i}{\sqrt{3}} & 0 \\
 0 & -i \\
\end{array}
\right]\right).
    \end{equation}
Then a direct computation using exact arithmetic verifies that the self-adjoint tuple $X$ defined in equation \eqref{eq:FreeExForm} is irreducible over $\CC$ and is a free extreme point at level four of $\cD_{F^{[3]}}$. More precisely, \cite[Theorem 1.1 (3)]{EHKM} and \cite[Lemma 2.1 (3)]{EH} together show that in order to claim that $X$ is free extreme, one need only verify that $X$ is irreducible and that the only $\beta \in M_{n \times 1} (\CC)^3$ that satisfies
\[
\Lambda_{F^{[3]}} (\beta) P_{F^{[3]},X} = 0
\]
is $\beta = 0$.  Here $P_{F^{[3]},X}$ is a matrix whose columns form a basis for the null space of $L_{F^{[3]}} (X)$. It is well-known that a tuple of self-adjoint matrices $X = (X_1,\dots,X_g) \in SM_{n} (\C)^g$ is irreducible if and only if the only matrices that commute with each $X_i$ are multiples of the identity. Thus each condition can be verified by solving a linear system of equations.

In the case that $n=6$, set $\alpha = \sqrt{2}-1$ and let
       \begin{equation}
    \label{eq:FreeExForm6}
   \hspace*{-.1 in} C = \frac{1}{4}\left(\left[
\begin{array}{ccc}
 \alpha +1 & 0 & -\alpha  \\
 0 & \alpha +1 & -\alpha  \\
 -\alpha  & -\alpha  & \alpha +1 \\
\end{array}
\right],\left[\begin{array}{ccc}
 -4 \sqrt{2\alpha } & 0 & 0 \\
 0 & 4  \sqrt{2\alpha } & 0 \\
 0 & 0 & 0 \\
\end{array}
\right],i\left[
\begin{array}{ccc}
 0 & 3 \alpha -1 & \alpha  \\
 3 \alpha -1 & 0 & \alpha  \\
 \alpha  & \alpha  & 3-\alpha  \\
\end{array}
\right]\right).
    \end{equation}
    Then $X$ defined as in equation \eqref{eq:FreeExForm} is a free extreme point at level six of $\cD_{F^{[3]}}$.
\end{proof}

We perform the computations appearing in the above proof in exact arithmetic using the NCSpectrahedronExtreme paclet \cite{EEdO+} for Mathematica. A publicly available Mathematica notebook containing the computations with detailed documentation for the steps is available at \href{https://github.com/NCAlgebra/UserNCNotebooks}{https://github.com/NCAlgebra/UserNCNotebooks}. Installation instructions for NCSpectrahedronExtreme are found in the NCSpectrahedronExtreme folder of the UserNCNotebooks directory. 

\begin{remark}
    The free extreme points constructed in Theorem \ref{thm:RealSpinBallFreeExtreme} are unitarily equivalent to real matrix tuples, thus there exist real matrix tuples in $\cD_{F^{[3]}}$ at levels 4 and 6 that are free extreme. For example, if one takes $C$ as in equation \eqref{eq:FreeExForm4} and $X$ as in equation \eqref{eq:FreeExForm} and defines 
    \[
    U=
\frac{\sqrt{2}}{2}\begin{bmatrix}
    I_2 & -i I_2 \\
    I_2 & i I_2
\end{bmatrix}
\qquad \mathrm{and} \qquad \alpha = 1+\frac{1}{\sqrt{3}},
    \]
    then $U$ is a unitary and 
    \[ \hspace*{-.15 in}
U^* X U = \frac{1}{2}\left(\left[
\begin{array}{cccc}
\alpha & 0 & 0 & 0 \\
 0 & 3\alpha-4 & 0 & 0 \\
 0 & 0 & -\alpha & 0 \\
 0 & 0 & 0 & -3\alpha+4 \\
\end{array}
\right],\left[
\begin{array}{cccc}
 0 & 1 & 0 & 0 \\
 1 & 0 & 0 & 0 \\
 0 & 0 & 0 & -1 \\
 0 & 0 & -1 & 0 \\
\end{array}
\right],
\left[
\begin{array}{cccc}
 0 & 0 & 3-2\alpha & 0 \\
 0 & 0 & 0 & 1 \\
 3-2\alpha& 0 & 0 & 0 \\
 0 & 1 & 0 & 0 \\
\end{array}
\right]\right).
    \]
    A similar choice of unitary works in the $n=6$ case.
\end{remark}

\begin{remark}
We prove a similar result for length $g > 3$ in Corollary \ref{cor:spinwrapup}, after some discussion about the extreme points of projections of matrix convex sets. Moreover, these results imply that the hypothetical posed directly after \cite[Theorem 5.12]{FarenickSpinOriginal}, about equality of the spin ball and the max ball for $g \geq 3$, has a negative answer. The notational conventions used in \cite{FarenickSpinOriginal} reverse the roles of the min and max monikers compared to our convention; see the text slightly before \cite[Proposition 14.6]{HKMS19}. 
\end{remark}

\begin{corollary}
    Modulo unitary equivalence, there are infinitely many free extreme points at levels $4$ and $6$ of $\cD_{F^{[3]}}$. 
\end{corollary}
\begin{proof}
Corollary \ref{cor:symmetry2} shows that $\cD_{F^{[3]}}$ is closed under orthogonal transformations. Orthogonal transformations respect the structure of direct sums, so free extreme points (that is, irreducible elements whose only dilations are direct sums) are also preserved under orthogonal transformations. It is not hard to see that orthogonal transformations of the listed extreme points produce matrices with differing eigenvalues, hence they cannot be unitarily equivalent. 
\end{proof}

\subsection{An open question about operator system properties}

Kavruk showed in the discussion directly before \cite[Corollary 10.12]{Kav14} that the $5$-dimensional operator system
\[ \cT_2 \,\, = \,\, \text{span}_{\mathbb{C}}(I_4, E_{12}, E_{21}, E_{34}, E_{43}) \,\, \subseteq \,\, M_4 \]
does not have the lifting property. That is, a UCP map from $\cT_2$ into a quotient $C^*$-algebra $\mathcal{A}/\mathcal{I}$ need not lift to a UCP map from $\cT_2$ into $\mathcal{A}$. This was established using the operator system dual: the lifting property for finite-dimensional operator systems is dual to a property called $1$-exactness (or exactness, depending on the authors) by \cite[Theorem 6.6]{Kav14}, and a $1$-exact operator system spanned by unitaries must generate an exact $C^*$-algebra by \cite[Corollary 9.6]{KPTT13}. For the case of $\cT_2$, \cite[Theorem 10.11]{Kav14} shows that the operator system dual of $\cT_2$ is isomorphic to the operator system $\cS_U$ spanned by the generating unitaries $U_1, U_2$ in the full group $C^*$-algebra $C^*(\mathbb{F}_2)$, which is a non-exact $C^*$-algebra. Hence, $\cS_U$ fails to be $1$-exact as in \cite[Corollary 10.13]{Kav14}, and its dual $\cT_2$ fails to have the lifting property. The same technique applies essentially verbatim to the free product $\mathbb{Z}/2 * \mathbb{Z}/2 * \mathbb{Z}/2$ in place of the free group, giving a $4$-dimensional operator system that fails to have the lifting property.

The fact that $\cT_2$ is a subsystem of the $4 \times 4$ matrices, specifically within $M_2 \oplus M_2$, yet fails the lifting property is particularly notable.  Choi and Effros showed that every separable nuclear $C^*$-algebra has the lifting property as part of a more general result \cite[Theorem 3.10]{ChoiEffros76}. Thus, Kavruk's example shows that, even for a \textit{finite-dimensional} operator system spanned by unitary \textit{matrices}, knowing the lifting property of the $C^*$-envelope is not sufficient to establish the lifting property of the operator system.

The lifting property and 1-exactness are tied to approximation properties within matrix convex sets, as studied in \cite{PP21}. First, operator systems with the lifting property correspond to matrix convex sets that are approximated either by free spectrahedra or by matrix convex sets that are maximal above a finite level by \cite[Theorem 3.3 and Corollary 3.4]{PP21}. Second, $1$-exact operator systems correspond to matrix convex sets that are approximated either by matrix ranges of matrix tuples or by matrix convex sets generated by a finite level by \cite[Theorem 3.7 and Corollary 3.9]{PP21}. These ideas also relate to measurements of completely bounded norms of $k$-positive maps, as in \cite[\S 6]{ADMPR}. Computing the dilation scales, Hausdorff distances, and completely bounded norms mentioned in all of these results is generally quite difficult, and there are not very many key examples from which to draw.

\begin{question}
Does the operator system spanned by $F_1, F_2, F_3$ have the lifting property? Equivalently, is the operator system associated to the matrix convex set $\cD_{F^{[3]}}$ $1$-exact? 
\end{question}

By the results of \cite{PP21}, if a negative answer to this open question is reached, this would lead to a much stronger claim than Theorem \ref{thm:RealSpinBallFreeExtreme}. Not only would $\cD_{F^{[3]}}$ fail to be a minimal matrix convex set over its first level, but it would fail to be generated by any one fixed level, or approximated in a uniform Hausdorff distance by such sets as the level grows. Note that the $C^*$-algebra generated by $F^{[3]}$ is also $M_2 \oplus M_2$, similar to Kavruk's example.

Another option is that the structure of the minimal and maximal convex sets themselves may pose an obstruction.

\begin{question}
Is the operator system associated to $\mathcal{W}^{\text{max}}(\overline{\mathbb{B}}_3)$ 1-exact? Equivalently, does the operator system associated to $\mathcal{W}^{\text{min}}(\overline{\mathbb{B}}_3)$ have the lifting property?
\end{question}

A negative answer to this open question would have a corollary that $\mathcal{W}^{\text{min}}(\overline{\mathbb{B}}_3)$ cannot be a free spectrahedron, or equivalently that $\mathcal{W}^{\text{max}}(\overline{\mathbb{B}}_3)$ cannot be expressed as the matrix range of a tuple of matrices.

\subsection{Matrix convex sets over the ball}

There are multiple known matrix convex sets over the Euclidean ball. For the self-adjoint case, a notable example is the matrix ball
\[ \mathcal{B}_g \,\, := \,\, \{ X \in SM^g: \sum X_j^2 \preceq I\}. \]
The matrix ball is discussed in \cite[\S 9]{DDOSS17}, without a classification of its free extreme points. For $g = 2$, the matrix ball $\mathcal{B}_2$ is considered in \cite[\S 7.2.1]{EHKM}, where it is called the \lq\lq wild disk,\rq\rq\hspace{0pt} and its free extreme points are classified. Here we provide a straightforward generalization of the classification of free extreme points to $g \geq 3$, with techniques that are also reminiscent of \cite[Theorem 2.8]{Pas22}.

\begin{proposition}
Let $X \in \mathcal{B}_g$, and define $S = \sum\limits_{j=1}^g X_j^2$, $V = \ker(I - S)$, and $W = V^\perp$. If $S = I$, then $X$ is an Arveson extreme point of $\mathcal{B}_g$. If $S \not= I$, then $X$ is an Arveson extreme point of $\mathcal{B}_g$ precisely when the operators
\[ R_j \,\, := \,\, P_V \, X_j|_W \]
are injective with linearly independent ranges. That is, it is impossible to find vectors $w_j \in W$, not all zero, with $\sum\limits_{j=1}^g X_j w_j \in W$.
\end{proposition}
\begin{proof}
In what follows, consider a dilation $Y_j = \begin{bmatrix} X_j & a_j \\ a_j^* & b_j \end{bmatrix}$, where $a_j$ is a column vector and $b_j$ is a real scalar, so that
\[ Q \, := \, \sum\limits_{j=1}^g Y_j^2 \,\, = \,\, \begin{bmatrix} S + \sum a_j a_j^* & \sum X_j a_j + a_j b_j \\ \sum a_j^* X_j + b_j a_j^* & \sum a_j^* a_j + b_j^2 \end{bmatrix}. \]
The positive matrix $Q$ must be a contraction in order for $Y \in \mathcal{B}_g$.

If $S = I$, then the top-left block immediately shows that each $a_j = 0$, so the dilation is trivial and $X$ is an Arveson extreme point. Similarly, if $S \not= I$, then since $S$ acts as the identity on $V$, we must have $a_j \in V^\perp = W$. Since the identity on $V$ cannot nontrivially dilate to a contraction, we must have $P_V( \sum X_j a_j + a_j b_j) = 0$, which means that $P_V(\sum X_j a_j) = 0$, or equivalently $\sum X_j a_j \in W$. As such, if a selection of $a_j \in W$ (not all zero) meeting this criterion does not exist, then $X$ is an Arveson extreme point.

Conversely, if $a_1, \ldots, a_g \in W$ (not all zero) may be found meeting $\sum X_j a_j \in W$, then we may select $b_j = 0$ in the dilation $Y_j$, and compute
\[ Q \, = \, \begin{bmatrix} S + \sum a_j a_j^* & \sum X_j a_j \\ \sum a_j^* X_j & \sum a^*_j a_j \end{bmatrix}. \]
On the reducing subspace $V \oplus \{0\} \subseteq \mathbb{C}^n \oplus \mathbb{C}$, $Q$ acts as the identity. On the complement $W \oplus \mathbb{C}$ (certainly also reducing), we note that $P_WS|_W \preceq (I - \varepsilon)I$ and the other terms scale down with the $a_j$, so we may choose $a_j$ sufficiently small that $Q$ is contractive. There exists a nontrivial dilation of $X$ in $\mathcal{B}_g$, and $X$ is not an Arveson extreme point.
\end{proof}

It is also not difficult to prove that $\cD_F$ always sits properly inside the matrix ball; we include the proof for completeness.

\begin{proposition}\label{prop:spin_inside_matrix}
For any $g \geq 2$, $\cD_F$ is a proper subset of $\mathcal{B}_g$. Equivalently, $\mathcal{B}_g^{\, \circ}$ is a proper subset of $\mathcal{W}(F)$.
\end{proposition}
\begin{proof}
Fix $g \geq 1$ and suppose $X = (X_1, \ldots, X_g) \in \cD_F$. Then the symmetry of Corollary \ref{cor:symmetry2} implies that $-I \preceq \sum\limits_{j=1}^g F_j \otimes X_j \preceq I$, and hence the $C^*$-norm identity gives that 
\[ \sum\limits_{m=1}^g I \otimes X_m^2  \,\, + \,\, \sum\limits_{j \not= k} F_j F_k \otimes X_j X_k  \,\, = \,\, \sum\limits_{m=1}^g I \otimes X_m^2 \, + \sum\limits_{1 \leq j < k \leq g} F_j F_k \otimes (X_j X_k - X_k X_j)  \,\, \preceq \,\, I.\]
Computing the average of this inequality and the inequality found by conjugating by $F_1 \otimes I$ yields 
\[ \sum\limits_{m=1}^g I \otimes X_m^2 \, + \sum\limits_{2 \leq j < k \leq g} F_j F_k \otimes (X_j X_k - X_k X_j)  \,\, \preceq \,\, I. \]
A similar averaging procedure with conjugations by $F_2 \otimes I, \ldots, F_{g-1} \otimes I$ in order will cancel all cross terms and show that
\[ \sum\limits_{m=1}^g I \otimes X_m^2 \,\, \preceq \,\, I.\]
Equivalently, $\sum\limits_{m=1}^g X_m^2 \, \preceq \, I$, and $X \in \mathcal{B}_g$. Finally, the free polar dual of the containment $\cD_F \subseteq \mathcal{B}_g$ shows that $\mathcal{B}_g^{\, \circ} \subseteq \cW(F)$.

To prove proper containment, consider a direct computation on the tuple $X = \frac{1}{\sqrt{g}} F$. Because $\sum\limits_{j=1}^g F_j \otimes F_j$ has $g$ as an eigenvalue by \cite[Lemma 7.23]{DDOSS17}, $\sum\limits_{j=1}^g F_j \otimes X_j$ has $\sqrt{g}$ as an eigenvalue, and hence $X \not\in \mathcal{D}_F$. However, $\sum\limits_{j=1}^g X_j^2 = \sum\limits_{j=1}^g \frac{1}{g}\, I = I$, so $X \in \mathcal{B}_g$.
\end{proof}

Combining \cite[Corollary 9.4]{DDOSS17} and \cite[Proposition 14.14]{HKMS19} with Proposition \ref{prop:spin_inside_matrix} and Theorem \ref{thm:RealSpinBallFreeExtreme} (see also Corollary \ref{cor:spinwrapup}) shows that for any $g \geq 2$,
\begin{equation} \mathcal{W}^{\text{min}}(\overline{\mathbb{B}_g}) \,\, \subseteq \,\, \mathcal{D}_F \,\, \subseteq \,\, \mathcal{B}_g \,\, \subseteq \,\, \mathfrak{D}_g \,\, \subseteq \,\, \mathcal{B}_g^{\, \circ} \,\, \subseteq \,\, \mathcal{W}(F) \,\, \subseteq \,\, \mathcal{W}^{\text{max}}(\overline{\mathbb{B}_g}), \end{equation}
where for $g = 2$ the first and last containments are equalities but the others are proper, and for $g \geq 3$ all of the containments are proper.

In the non self-adjoint case, there are multiple additional{\iffalse {\eric ?? I added well-studied because it sounded before like there aren't other free spectrahedron examples. There should be infinitely many examples by just using Matricial Hahn Banach to separate points of then add that to the pencil. Its not clear that any of these are interesting though and they certainly aren't well studied.} \color{red} I didn't like that especially considering the phrase well-studied already appears in the sentence. I don't think it's accurate to call them all well-studied. \fi} examples of matrix convex sets over the ball, such as the well-studied row contractions \cite{Zh14}, as well as the \lq\lq mixed row contractions\rq\rq\hspace{0pt} $\mathcal{M}(d, g)$ of \cite{Pas22}. Here $d$ and $g$ denote the number of non self-adjoint and self-adjoint variables, respectively. We also introduce another matrix convex set over the ball.

\begin{definition}
Fix $d \geq 1$. Let $\mathcal{Q}_d$ denote the matrix convex set in $d$ non self-adjoint variables defined by the condition that each tuple $(T_1, T_2, \ldots, T_d) \in \mathcal{Q}_d$ satisfies
\[ \| \lambda_1 T_1 + \lambda_2 T_2 + \ldots + \lambda_d T_d \| \, \leq \, 1 \]
for every $\lambda = (\lambda_1, \lambda_2, \ldots, \lambda_d) \in \mathbb{C}^d$ with $|\lambda_1|^2 + |\lambda_2|^2 + \ldots + |\lambda_d|^2 = 1$.
\end{definition}

\begin{theorem}
The Cuntz isometries are maximal, infinite-dimensional elements of $\mathcal{Q}_d$.
\end{theorem}
\begin{proof}
Since a Cuntz tuple $(T_1, T_2, \ldots, T_d)$ satisfies $T_1T_1^* + T_2T_2^* + \ldots + T_d T_d^* = I = T_j^*T_j$ and $T_i^* T_j = 0$ for $i \not= j$, membership of $(T_1, T_2, \ldots, T_d)$ in $\mathcal{Q}_d$ is established by the C*-norm identity applied to 
\[ (\overline{\lambda_1} T_1^* + \overline{\lambda_2} T_2^* + \ldots + \overline{\lambda_d}T_d^*)(\lambda_1 T_1 + \lambda_2 T_2 + \ldots + \lambda_d T_d) \,\, = \,\, (|\lambda_1|^2 + |\lambda_2|^2 + \ldots + |\lambda_d|^2)I. \]
Now, consider an arbitrary dilation $S = \left( \begin{bmatrix} T_i & A_i \\ B_i & C_i \end{bmatrix} \right)_{i=1}^d \in \mathcal{Q}_d$. Define
\[ R_\lambda \,\, := \,\, (\overline{\lambda_1} S_1^* + \overline{\lambda_2} S_2^* + \ldots + \overline{\lambda_d}S_d^*)(\lambda_1 S_1 + \lambda_2 S_2 + \ldots \lambda_d S_d) \]
for $\lambda \in \mathbb{C}^d$ with $|\lambda_1|^2 \, + \, |\lambda_2|^2 + \ldots + |\lambda_d|^2 = 1$, and note that the top-left block of $R_\lambda$ is 
\[ \sum\limits_{i=1}^d |\lambda_i|^2(T_i^*T_i + B_i^*B_i) + \sum\limits_{i < j} \left( \overline{\lambda_i} \lambda_j (T_i^*T_j + B_i^*B_j) + \lambda_i \overline{\lambda_j} (T_j^*T_i + B_j^* B_i) \right)\]
\[ = \sum\limits_{i=1}^d |\lambda_i|^2 I \,\, + \,\, \sum\limits_{i=1}^d |\lambda_i|^2 B_i^*B_i  \,\, + \sum\limits_{i < j} \left( \overline{\lambda_i} \lambda_j B_i^*B_j + \lambda_i \overline{\lambda_j} B_j^* B_i \right)\]
\[ =  I \,\, + \,\, \sum\limits_{i=1}^d |\lambda_i|^2 B_i^*B_i \,\, + \sum\limits_{i < j} \left( \overline{\lambda_i} \lambda_j B_i^*B_j + \lambda_i \overline{\lambda_j} B_j^* B_i \right),\]
which must remain contractive for all choices of $\lambda$. An averaging procedure using the terms $(\pm \lambda_1, \lambda_2, \ldots, \lambda_g)$, then $(\lambda_1, \pm \lambda_2, \lambda_3, \ldots, \lambda_g)$, and so on will cancel the cross terms and show that
\[ I + \sum\limits_{i=1}^d |\lambda_i|^2 B_iB_i^* \] 
is also contractive. Restricting to the case when all entries of $\lambda$ are nonzero immediately implies that each $B_i = 0$. So,
\[ S_i \, = \, \begin{bmatrix} T_i & A_i \\ 0 & C_i \end{bmatrix} \]
and the top-left block of $R_\lambda$ is $I$. Next, the top-right block of $R_\lambda$ (which must be $0$ because $R_\lambda$ is a contraction) is 
\[ \sum\limits_{i=1}^d |\lambda_i|^2 T_i^* A_i \,\, + \,\, \sum\limits_{i < j} \left( \overline{\lambda_i} \lambda_j T_i^*A_j + \lambda_i \overline{\lambda_j} T_j^* A_i \right) \,\,\, = \,\,\, 0.\]
Evaluating at choices of $\lambda$ that vanish in all but one coordinate shows that $T_i^* A_i = 0$ for all $i$. Then, for any $\lambda$ that is nonzero only in the two coordinates $i < j$, we have $\overline{\lambda_i} \lambda_j T_i^*A_j + \lambda_i \overline{\lambda_j} T_j^* A_i = 0$. We may consider when $\lambda_i$ is real and $\lambda_j$ is imaginary, then the reverse, to reach two linear equations that together imply $T_i^* A_j = 0 = T_j^* A_i$. Finally, 
\[ \forall \, i, j \in \{1, \ldots, d\}, \,\,\, T_i^* A_j = 0.\]
It follows that for $j \in \{1, \ldots, d\}$, 
\[ A_j \,\, = \,\, (T_1 T_1^* + T_2 T_2^* + \ldots + T_d T_d^*) A_j \,\, = \,\, T_1(0) + \ldots + T_d(0) \,\, = \,\, 0,\]
which proves that the dilation $(S_1, \ldots, S_d)$ of $(T_1, \ldots, T_d)$ was trivial.
\end{proof}

However, the Cuntz isometries clearly cannot be the only maximal elements. 

\begin{corollary}
The adjoints of the Cuntz isometries are maximal, infinite-dimensional elements of $\mathcal{Q}_d$.
\end{corollary}
\begin{proof}
The set $\mathcal{Q}_d$ is closed under the coordinatewise adjoint $T_i \mapsto T_i^*$.
\end{proof}

\section{Free Spectrahedrops}\label{sec:drops}

In this section we study projections of free spectrahedra and their extreme points, with a particular emphasis on the role of the real and complex fields. We begin by showing that a particularly nice property of the extreme points of the first level of a real free spectrahedron is preserved under projections. 
\begin{proposition}
\label{prop:RealDropEucAreFreeLevel1}
Fix $g \leq h$ and let $\cD_A \subseteq SM(\C)^h$ be a bounded free spectrahedron. Additionally assume that $\cD_A (2)$ is closed under complex conjugation. Then the Euclidean extreme points of $\cP_g \cD_A (1)$ are free extreme points of $\cP_g \cD_A$. 
\end{proposition}

\begin{proof}
    The fact that a free extreme point is Euclidean extreme always holds, so we need only show that if $X$ is Euclidean extreme then it is free extreme. It is sufficient to show that if $X \in \cP_g \cD_A (1)$ is not free extreme, then it is not Euclidean extreme. If $X$ is not free extreme in $X \in \cP_g \cD_A (1)$, then there exist a nonzero $\beta \in M_{n \times 1} (\CC)^{g}$ and a $\gamma \in \RR^g$ such that
    \[
    \begin{bmatrix}
    X & \beta \\
    \overline{\beta} & \gamma
    \end{bmatrix} \in \cP_g \cD_A(2),
    \]
    from which it follows that there exists a tuple 
    \begin{equation}
    \label{eq:dilationInDrop}
    \begin{bmatrix}
    Y & \eta \\
   \overline{\eta} & \sigma
    \end{bmatrix}  \in SM_2 (\C)^{h-g} \qquad \mathrm{s.t.} \qquad  \left( \begin{bmatrix}
    X & \beta \\
    \overline{\beta} & \gamma
    \end{bmatrix},
     \begin{bmatrix}
    Y & \eta \\
    \overline{\eta} & \sigma
    \end{bmatrix} \right)\in \cD_A (2).
    \end{equation}
    Using the assumption that $\cD_A(2)$ is closed under complex conjugation and the fact that a free spectrahedron is closed under unitary conjugation, we now obtain that
    \[
\left( \begin{bmatrix}
    X & \overline{\beta} \\
    \beta & \gamma
    \end{bmatrix},
     \begin{bmatrix}
    Y & \overline{\eta} \\
    \eta & \sigma
    \end{bmatrix} \right),\left( \begin{bmatrix}
    X & -i\beta \\
    i\overline{\beta} & \gamma
    \end{bmatrix},
     \begin{bmatrix}
    Y & i\eta \\
    -i\overline{\eta} & \sigma
    \end{bmatrix} \right),\left( \begin{bmatrix}
    X & -i\beta \\
    i\overline{\beta} & \gamma
    \end{bmatrix},
     \begin{bmatrix}
    Y & -i\eta \\
    i\overline{\eta}& \sigma
    \end{bmatrix} \right)\in \cD_A (2).
    \]
    Next, using the convexity of $\cD_A(2)$ gives that
    \[
    \left(\begin{bmatrix}
    X & \mathrm{Re}(\beta) \\
    \mathrm{Re}(\beta) & \gamma
    \end{bmatrix},
     \begin{bmatrix}
    Y & \mathrm{Re}(\eta)\\
    \mathrm{Re}(\eta) & \sigma
    \end{bmatrix} \right), \left(\begin{bmatrix}
    X & \mathrm{Im}(\beta) \\
    \mathrm{Im}(\beta) & \gamma
    \end{bmatrix},
     \begin{bmatrix}
    Y & \mathrm{Im}(\eta)\\
    \mathrm{Im}(\eta) & \sigma
    \end{bmatrix} \right) \in \cD_A(2).
    \]
    As $\beta$ is nonzero, at least one of $\mathrm{Re}(\beta)$ and $\mathrm{Im}(\beta)$ is nonzero. Thus, we can without loss of generality additionally assume that $(\beta,\eta)$ is real valued and that $\beta$ is nonzero.

    Equation \eqref{eq:dilationInDrop} now implies that $ L_A (X,Y) \succeq 0$ and moreover that
    \[
    \ker L_A (X,Y) \subseteq \ker \Lambda_A (\beta,\eta).
    \]
    Furthermore, since $(\beta,\eta)$ is real, it follows that there is an $\alpha > 0$ such that 
    \[
L_{(A,B)} (X\pm \alpha \beta ,Y \pm \alpha \eta) \succeq 0,
    \]
    hence $X \pm \alpha \beta \in \cP_g \cD_{A} (1)$. As $\beta$ is nonzero, we conclude that $X$ is not Euclidean extreme in $\cP_2 \cD_{A} (1)$.
\end{proof}

\begin{remark}
    Proposition \ref{prop:RealDropEucAreFreeLevel1} is a direct generalization of \cite[Proposition 6.1]{EHKM}. In particular, the result emphasizes that projections of real free spectrahedra must have a reasonable number of free extreme points and that the matrix convex hull of those free extreme points recovers at least the first level of the set. 
\end{remark}

Using the above proposition shows that the free polar dual of a bounded real free spectrahedron can fail to be a real free spectrahedrop, in contrast with \cite[\S 4]{HKM17} for the complex case.

\begin{corollary}\label{cor:drop_polyhedron_bounds}
   Let $A \in SM_d(\RR)^g$ and assume that $\cD_A$ is a bounded free spectrahedron. If $\cD_A^\circ$ is the projection of a free spectrahedron whose second level is closed under complex conjugation, then $\cD_A (1)$ must be a polyhedron. 

   Furthermore, in this case, there must exist a unitary $U$ and a tuple of diagonal matrices $B \in SM_{d_1} (\RR)^g$ with $d_1 \geq g+1$ such that $U^* A U = B\oplus C \in SM_d (\RR)^g$ where $C$ is some self-adjoint matrix tuple of matrix size $d-d_1$.
\end{corollary}

\begin{proof}
    Note that if one works over $\C$ and does not assume closure under complex conjugation, then it is shown by \cite{HKM17} that a bounded free spectrahedrop is the projection of a bounded free spectrahedron. However, the proof is by construction, and the second level of the resulting bounded free spectrahedron is not closed under complex conjugation. Thus, to appeal to Proposition \ref{prop:RealDropEucAreFreeLevel1}, we must first show that if $\cD_A^\circ$ is the projection of a (possibly unbounded) free spectrahedron whose second level is closed under complex conjugation, then it is the projection of a bounded free spectrahedron whose second level is closed under complex conjugation. 
    
    To this end, suppose there exists an $h \in \mathbb{N}$ and $B \in SM (\C)^{g+h}$ such that $\cD_A^\circ = \mconv(A) = \cP_g (\cD_B)$ and such that $\cD_B(2) = \overline{\cD_B(2)}$. Since $A \in \cP_g (\cD_B)$, there is some $X \in SM_d(\C)^h$ with $(A,X) \in \cD_B$. Let $C \in SM(\C)^{g+h}$ be any tuple such that $\cD_C$ is bounded, $\cD_C (2) = \overline{\cD_C(2)}$, and $(A,X) \in \cD_C$. For example, one may take $C = \alpha F^{[g+h]}$ for an appropriate choice of $\alpha >0$. It is straightforward to verify that $\cD_B \cap \cD_C = \cD_{B \oplus C}$ is a bounded free spectrahedron whose second level is closed under complex conjugation. Moreover, since $A \in \cP_g \cD_{B \oplus C}$, we obtain
  \[
    \mconv(A) \subseteq \cP_g \cD_{B \oplus C} \subseteq \cP_g \cD_B = \mconv(A).
  \]
  Therefore, $\cD_A^\circ = \mconv(A) = \cD_{B \oplus C}$, which proves our initial claim.

  Now, WLOG assume $A = \oplus A^i$ where each $A^i \in SM_{d_i} (\RR)^g$ is an irreducible tuple of matrices. Then up to unitary equivalence, the free extreme points of $\cD_A^\circ$ are precisely the $A^i$. Now, if $\cD_A (1)$ is not a polyhedron, then $\cD_A^\circ (1)$ is also not a polyhedron, hence $\cD_A^\circ (1)$ has infinitely many Euclidean extreme points. It follows that not every Euclidean extreme point of $\cD_A^\circ (1)$ is free extreme. Thus, as a consequence of Proposition \ref{prop:RealDropEucAreFreeLevel1}, $\cD_A^\circ$ cannot be the projection of a bounded free spectrahedron whose second level is closed under complex conjugation. 

  In the furthermore statement, the fact that $\cD_A^\circ (1) = \mconv (B)$ if $\cD_A^\circ$ is the projection of a complex conjugation closed free spectrahedron is immediate from the first part of the argument. The fact that $d_1 \geq g+1$ follows from the fact that a bounded polyhedron in $g$ variables that has interior must have at least $g+1$ extreme points. 
\end{proof}

\begin{example}
As in the previous section, let
\[ F^{[2]} = \left(\begin{bmatrix}
        1 & 0 \\
        0 & -1
    \end{bmatrix},\begin{bmatrix}
        0 & 1 \\
        1 & 0
    \end{bmatrix}
    \right), \,\,\,\,\,\,\,\,\,\,\,\,\,\,
    P = \left(\begin{bmatrix}
        1 & 0 \\
        0 & -1
    \end{bmatrix},\begin{bmatrix}
        0 & 1 \\
        1 & 0
    \end{bmatrix},
    \begin{bmatrix}
        0 & i \\
        -i & 0 
    \end{bmatrix}
    \right),\]
and recall that $\cD_{F^{[2]}}^{\circ} = \cW(F^{[2]}) = \cW^{\max}(\overline{\mathbb{D}})$ is also equal to $\cP_2 \cD_P$ by Proposition \ref{prop:max_proj_pauli}. One may use the method discussed in the proof of \cite[Proposition 4.14]{HKM17} to compute this. In particular, given a bounded free spectrahedron $\cD_A$, \cite{HKM17} shows how to express $\cD_A^\circ$ as the projection of a (complex) free spectrahedron.

The above expresses a free spectrahedrop that is closed under complex conjugation as the projection of a free spectrahedron, but the latter is not closed under complex conjugation. Corollary \ref{cor:drop_polyhedron_bounds} demonstrates that no matter how $A$ is chosen, if $\cP_2 \cD_A = \cW^{\max}(\overline{\mathbb{D}})$, then $\cD_A$ must not be closed under complex conjugation.
\end{example}

\begin{proposition}
\label{proposition:MatExAreProjOfMatEx}
    Let $K \subset SM(\C)^h$ be a closed bounded matrix convex set and fix $g \leq h$. Then the matrix (Euclidean) extreme points of $\cP_g K$ are projections of matrix (Euclidean) extreme points of $K$. In notation,
    \[
    \matex \cP_g K \subset \cP_g \matex  K \qquad \mathrm{and} \qquad \Euc \cP_g K \subset \cP_g \Euc K.
    \]
\end{proposition}

\begin{proof}
    Let $X \in \matex \cP_g K(n)$. Then there exists a tuple $Y \in SM_n (\C)^{h-g}$ such that $(X,Y) \in K(n)$. Using \cite[Theorem 6.8]{Kr}, there exist matrix extreme points $(X^j,Y^j) \in K (n_j)$ for $j=1,\dots,\ell$ with $n_j \leq n$ for each $j$ and matrices $V_j \in M_{n_j \times n} (\C)$ such that
    \[
    (X^j,Y^j) = \sum_{j=1}^\ell V_j^* (X^j,Y^j) V_j \qquad \mathrm{and} \qquad \sum_{j=1}^\ell V_j^* V_j = I_n.
    \]
    It immediately follows that $X^j \in \cP_g K(n_j)$ with $n_j \leq n$ for each $j$ and moreover that
    \[
    X^j = \sum_{j=1}^\ell V_j^* X^j V_j \qquad \mathrm{and} \qquad \sum_{j=1}^\ell V_j^* V_j = I_n.
    \]
    The assumption that $X$ is matrix extreme then implies that for each $j$, there is a unitary $U_j$ such that $X = U_j^* X^j U_j$. 

    Since $\matex K$ is closed under unitary conjugation, we then obtain that 
    \[
    U_j^* (X^j,Y^j) U = (X,U_j^* Y^j U_j) \in \matex K.
    \]
    We conclude that $X$ is a projection of a matrix extreme point of $K$. The proof that Euclidean extreme points are projections of Euclidean extreme is similar. 
\end{proof}

On the other hand, free extreme points of $\cP_g K$ need not be projections of free extreme points of $K$.

\begin{example}
Let $K$ be any closed and bounded matrix convex set with no free extreme points (such as in \cite{Eve18} or in \cite[Example 6.30]{Kr}), and examine $\cP_1 K$. There is only one matrix convex set over an interval, and there are free extreme points of $\cP_1 K$. Thus, free extreme points of $\cP_1 K$ cannot be projections of free extreme points of $K$.
\end{example}

\begin{example}
Let $K = \cW^{\max}(\overline{\mathbb{D}})$ be the maximal matrix convex set with first level equal to the ball in two variables. In the language of the previous section, 
\[
    K = \cW(F^{[2]}) = \mconv(F^{[2]}).
\]
Once again, consider $\cP_1 K$, which has free extreme points (namely $\pm 1$) of matrix dimension one. These are not projections of free extreme points of $K$, since the only free extreme point of $K$ is $F^{[2]}$ itself, which has matrix dimension two.

This example shows that adding an assumption that $K$ is the matrix convex hull of its free extreme points is not sufficient to prove that free extreme points are projections of free extreme points. However, note that both $\pm 1$ are compressions of the first entry of $F^{[2]} = \left(\begin{bmatrix} 1 & 0 \\ 0 & -1 \end{bmatrix}, \begin{bmatrix} 0 & 1 \\ 1 & 0 \end{bmatrix}\right)$, and in that first entry alone, the dilation $\begin{bmatrix} 1 & 0 \\ 0 & -1 \end{bmatrix}$ is a trivial dilation (direct sum). This provides context for the next result.
\end{example}

\begin{theorem}\label{theorem:FreeExtremePartialExtension}
Suppose $K$ is the matrix convex hull of its free extreme points (for instance, if $K$ is a free spectrahedron that is closed under complex conjugation), and let $X \in \cP_g K$ be a free extreme point of $\cP_gK$. Then either $X$ is the projection of a free extreme point of $K$, or there exists some $Z \in \cP_g K$ such that $X \oplus Z$ is the projection of a free extreme point of $K$.
\end{theorem}
\begin{proof}
Suppose that $K$ is the matrix convex hull of its free extreme points. By (\ref{eq:arv_equivalence}), any point of $K$ dilates to an Arveson extreme point. If $X \in \cP_g K$, then by definition there exists some $Y$ such that $(X,Y) \in K$, hence the tuple $(X,Y)$ dilates to some Arveson extreme point $(\hat{X},\hat{Y}) \in K$. Since this means $X$ admits a dilation to $\hat{X} \in P_g K$, the fact that $X$ is free extreme in $\cP_g K$ means that $\hat{X}$ is either unitarily equivalent to $X$ or to $X \oplus A$ for some $A \in \cP_g K$.

We now address the issue of irreducibility, as $(\hat{X}, \hat{Y})$ is an Arveson extreme point but might not be irreducible. Denote the minimal reducing spaces of $\hat{X}$ as $V_1, \ldots , V_n$, where $V_1$ corresponds to $X$, and the other $V_j$ (if any) correspond to the irreducible summands of $A$. Now, consider the minimal reducing subspaces $W_1, \ldots W_m$ of $(\hat{X}, \hat{Y})$, which form an orthogonal spanning set. Each $W_i$ is a reducing subspace for $\hat{X}$, so each $W_i$ decomposes as a sum of different $V_j$. One of the $W_i$ either equals $V_1$ or includes $V_1$ in its sum. Projecting onto this $W_i$ shows that either a tuple $(X, Q)$ or a tuple $(X \oplus Z, Q)$ is unitarily equivalent to an irreducible summand of $(\hat{X}, \hat{Y})$, which is a free extreme point.
\end{proof}

While free extreme points in a projection do not necessarily extend to free extreme points, Theorem \ref{theorem:FreeExtremePartialExtension} is more than sufficient to extend our results for $\cD_{F^{[3]}}$ to more than three variables.

\begin{corollary}\label{cor:spinwrapup}
Fix $g \geq 3$. Then $\cD_{F^{[g]}}$ has a free extreme point of matrix dimension at least $6$, and is consequently not equal to the minimal matrix convex set over the Euclidean ball.
\end{corollary}
\begin{proof}
If $g = 3$, simply use Theorem \ref{thm:RealSpinBallFreeExtreme}. If $g > 3$, then by Corollary \ref{cor:extending_by_zero}, $\cP_3 \cD_{F^{[g]}} = \cD_{F^{[3]}}$, where the latter set has a free extreme point $X$ of matrix dimension $6$. Theorem \ref{theorem:FreeExtremePartialExtension} produces a free extreme point of $\cD_{F^{[g]}}$ of the form $E = (X, Y)$ or $E = (X \oplus Z, Y)$.
\end{proof}

Corollary \ref{cor:drop_polyhedron_bounds} illustrates that it is rare for the free polar dual of a real free spectrahedron to be the projection of a real free spectrahedron. However, there are cases where this does occur. To prove this, we use the following lemma. 

\begin{lemma}
    \label{lemma:ProjOfMinimalIsMinimal}
    The projection of a minimal matrix convex set is a minimal matrix convex set. 
\end{lemma}
\begin{proof}
    Straightforward.
\end{proof}

\begin{definition}
A free polyhedron is a free spectrahedron $\cD_A$ whose coefficient tuple $A$ is diagonal. Equivalently, a free polyhedron is defined as $\cW^{\mathrm{max}}(K)$, where $K$ is a polyhedron.
\end{definition}

We emphasize that a free polyhedron is maximal over its first level; a free spectrahedron whose first level happens to be a polyhedron is not necessarily a free polyhedron. This terminology choice comes from the fact that a polyhedron is precisely the set of points that satisfy a finite collection of affine linear inequalities, while a free polyhedron is the collection of matrix tuples that satisfy a finite collection of affine linear inequalities. 

\begin{proposition}
\label{prop:PolyDualIsProjOfSimplex}
        Let $\cD_A$ be a bounded free polyhedron. Then $\cD_A^\circ$ is the projection of a real free spectrahedron. Moreover, $\cD_A^\circ$ is the projection of a free simplex. 
\end{proposition}
\begin{proof}
    If $\cD_A$ is a bounded free polyhedron, then $\cD_A^\circ$ is the minimal matrix convex set whose first level is $\cD_A^\circ (1)$, which is a bounded polyhedron. It is well-known that every bounded polyhedron is the projections of a simplex and a straightforward exercise shows that a bounded polyhedron containing $0$ in its interior is the coordinate projection of a simplex containing 0 in its interior. It follows that there exists some bounded free simplex $\cD_B$ such that $\cP_g \cD_B(1) = \cD_A^\circ (1)$. As a free simplex is a minimal matrix convex set, the result now follows from Lemma \ref{lemma:ProjOfMinimalIsMinimal}.
\end{proof}

\begin{remark}
    The projection of a maximal matrix convex set may not be maximal. As free simplices are both minimal and maximal \cite[Theorem 4.7]{FNT}, this is a quick consequence of Lemma \ref{lemma:ProjOfMinimalIsMinimal}. For example, if $K$ is minimal matrix convex set generated by the unit square in two variables, then $K$ is a projection of a free simplex, which is a maximal matrix convex set. However, $K$ is not a maximal matrix convex set. A similar argument shows that the maximal matrix convex set generated by the unit square in two variables cannot be the projection of a free simplex, illustrating a notable difference between the classical and free settings. 
\end{remark}

\subsection{Free spectrahedrops without free extreme points}

We now turn to constructing a complex conjugation closed free spectrahedrop that has no free extreme points. We begin by looking at the extreme points of the matrix convex hull of a finite union of matrix convex sets. 

\begin{theorem}
\label{theorem:UnionClosedAndExtreme}
Let $\{K_k\}_{k=1}^\ell$ be a finite collection of closed bounded matrix convex sets. Then
\[
K:=\mconv\left(\cup_{k=1}^\ell K_k\right)
\]
is a closed bounded matrix convex set. Furthermore, the matrix (free) extreme points of $K$ are contained in the union of the matrix (free) extreme points of the $K_k$. In notation,
\[
\matex K \subseteq \cup_{k=1}^\ell \matex K_k \qquad \mathrm{and} \qquad \free K \subseteq \cup_{k=1}^\ell \free K_k
\]
\end{theorem}
\begin{proof}
    We first show that for fixed $n \in \mathbb{N}$, one has
    \[
    K(n) =  \Big(\mconv\left(\cup_{k=1}^\ell K_k(n)\right)\Big)(n)
    \]
    To this end, let $X \in K(n)$. Then by definition, there exist a finite collection of tuples $\{X^j\}_{j=1}^m \subset \cup_{k=1}^\ell K_k$ and matrices $V_j \in M_{n_j \times n} (\CC)$ such that
    \[
    X= \sum_{j=1}^m V_j^* X^j V_j \qquad \mathrm{and} \qquad \sum_{j=1}^m V_j^* V_j = I.
    \]
    Now, for each $j$, let $W_j \in M_{n_j \times n}$ be an isometry range containing the range of $V_j$ so that $W_j W_j^* V_j = V_j$. Then $W_j^* X^j W_j \in \cup_{k=1}^\ell K_k(n)$ for each $j$, and 
    \[
    X= \sum_{j=1}^m V_j^* W_j (W_j^*X^jW_j) W_j^* V_j \qquad \mathrm{and} \qquad \sum_{j=1}^m V_j^*W_j W_j^* V_j = I,
    \]
    so $X$ is a matrix convex combination of elements of $\cup_{k=1}^\ell K_k(n)$, as claimed.

    The fact that $K$ is closed and bounded is now an immediate consequence of \cite[Corollary 2.5]{HL}. More precisely, \cite[Corollary 2.5]{HL} shows that $\mconv\left(\cup_{k=1}^\ell K_k(n)\right)$ is level-wise closed and bounded, hence $K(n)$ is closed and bounded. 

    We next consider the matrix extreme points of $K$. To this end, suppose 
    \[
    X \in K \backslash \left(\cup_{k=1}^K \matex K_k\right).
    \]
    Then, using the above together with \cite[Theorem 6.8]{Kr}, $X$ can be expressed as matrix convex combination of elements of $\cup_{k=1}^\ell \matex K_k(n)$. However, since $X$ is not a matrix extreme point of any of the $K_k$, this expresses $X$ as a matrix convex combination of elements of $K(n)$ that $X$ is not unitarily equivalent to. Thus, $X$ is not matrix extreme in $K(n)$.

    Finally, suppose that $X$ is free extreme in $K$. Then $X$ is also matrix extreme in $K$, so from the above it is matrix extreme in $K_k$ for some $k$. If $X$ is not free extreme in $K_k$, then it nontrivially dilates in $K_k$, hence it nontrivially dilates in $K$, which would contradict the assumption that $X$ is free extreme in $K$. We conclude that $X$ is free extreme in $K_k$. 
\end{proof}

\begin{remark}
The Euclidean extreme points of a union of free spectrahedra need not be Euclidean extreme points of sets in the union. A simple example is obtained by letting $K_1,K_2,K_3 \subset SM(\C)^2$ be the free intervals
\[
\begin{split}
    K_1 & = \mconv (\{(-2,1),(1,1)\}) \\
    K_2 & = \mconv (\{(1,1),(1,-2)\})  \\
    K_3 & = \mconv (\{(-2,1),(1,-2)\}). \\
\end{split}
\]
It is straightforward to verify that $\mconv (K_1 \cup K_2 \cup K_3)$ is the free simplex with extreme points $(-2,1),(1,1)$, and $(1,-2)$, which has defining tuple 
\[
A = \big(\mathrm{diag}(1,0,-1),\mathrm{diag}(0,1,-1) \big).
\]
By checking the conditions of \cite[Corollary 2.3 (2)]{EHKM}, one can show that the tuple 
\[
 X = \left(\begin{bmatrix} 1 & 0 \\ 0 & 0 \end{bmatrix},
    \begin{bmatrix}
        \frac{1}{2} & \sqrt{\frac{5}{6}} \\ 
        \sqrt{\frac{5}{6}} & -\frac{2}{3}
    \end{bmatrix}\right)
\]
is a Euclidean extreme point of $\cD_A$. On the other hand, since $\cD_A$ is the matrix convex hull of its first level, $X$ cannot be matrix extreme. Furthermore, $X \notin K_1 \cup K_2 \cup K_3$. To see this, observe that if $X \in K_j$ for some $j$, then $\mconv(X) \subset K_j.$
However,
\[
(0,-2/3) \in \mconv(X) \qquad\mathrm{and}\qquad (0,-2/3) \notin K_j \ \mathrm{for\ any\ } j=1,2,3.
\]

The example can easily be modified to use free spectrahedra instead of general matrix convex sets. For example, let
\[
\begin{split}
    K_1' & = \mconv (\{(-2,1),(1,1),(1/10,1/10)\}) \\
    K_2' & = \mconv (\{(1,1),(1,-2),(-1/10,0)\})  \\
    K_3' & = \mconv (\{(-2,1),(1,-2),(0,-1/10)\}). \\
\end{split}
\]
Then $K_1',K_2',$ and $K_3'$ are free simplices containing $0$ in their interior, hence are free spectrahedra. With this modification, we still have $\mconv(\cup_{j=1}^3 K_j) = \cD_A$. One can then verify that $X \notin K_1 \cup K_2 \cup K_3$ either by evaluating $X$ in defining pencils for the $K_j'$ or by showing that $\mconv(X)(1) \not\subseteq K_j'(1)$ for each $j$. 
\end{remark}

\begin{corollary}
\label{corollary:SpectrahedropWithNoFreeExtreme}
    There exist closed bounded free spectrahedrops that are closed under complex conjugation and have no free extreme points. 
\end{corollary}
\begin{proof}
    As seen in \cite[Example 6.30]{Kr} and \cite[Remark 2.5 and Corollary 2.11]{Pas22}, there exist closed bounded complex free spectrahedra (which are not closed under complex conjugation) that have no free extreme points. Let $\cD_A$ be such a free spectrahedron, and note that $\overline{\cD_A}= \cD_{\overline{A}}$ is also a closed bounded free spectrahedron with no free extreme points. 
    
    Now, set $K = \mconv(\cD_A \cup \overline{\cD_A})$. Then from Theorem \ref{theorem:UnionClosedAndExtreme}, $K$ is a closed bounded matrix convex set. Moreover, since $\cD_A$ and $\overline{\cD_A}$ have no free extreme points, $K$ cannot have free extreme points. On the other hand, it is straightforward to check that $K$ is closed under complex conjugation. Finally, as a consequence of \cite[Proposition 4.18]{HKM17}, $K$ is a free spectrahedrop. 
\end{proof}

    As a consequence of Proposition \ref{prop:RealDropEucAreFreeLevel1}, if $K$ is a bounded free spectrahedrop that has no free extreme points, then $K$ cannot be the projection of a bounded free spectrahedron that is closed under complex conjugation. This leads to the following question. 

\begin{question}
    Let $\cD_A$ be a bounded free spectrahedron that is closed under complex conjugation. Is a projection of $\cD_A$ the matrix convex hull of its free extreme points? More generally, suppose $K$ is a closed bounded matrix convex set that is the matrix convex hull of its free extreme points. Are projections of $K$ the matrix convex hull of their free extreme points?
\end{question}

\subsection{Free spectrahedra whose dual are free spectrahedra}

As a generalization of Theorem \ref{thm:Paulidual}, we now present a class of complex free spectrahedra whose free polar duals are free spectrahedra.

\begin{theorem}
\label{theorem:DualFreeSpecs}
Fix $d \in \mathbb{N}$ and let $A \in SM_{d} (\CC)^{d^2-1}$ be chosen so that $\cD_A$ is a bounded free spectrahedron. Then there exists a tuple $B \in SM_{d} (\CC)^{d^2-1}$ such that $\cD_A^\circ = \cD_B$. As a consequence, $\cD_A = \mconv(B)$. Thus, up to unitary equivalence, $B$ is the only free extreme point of $\cD_A$. Moreover, $\cD_A$ is the matrix convex hull of its free extreme points.  
\end{theorem}

\begin{proof}
 The assumption that $\cD_A$ is bounded implies that $\{I,A_1,\dots,A_{d^2-1}\}$ is linearly independent over $\RR$. Since the dimension of $SM_{d}(\CC)$ as a real vector space is $d^2$, it follows that  $\{I,A_1,\dots,A_{d^2-1}\}$ is a basis for $SM_{d} (\CC)$. As a consequence, the tuple $A$ must be irreducible over $\CC$. 
 
 Now, from \cite[Proposition 4.14]{HKM17}, $\cD_A^\circ$ is the projection of a bounded free spectrahedron $\cD_B$ for some $B \in SM_{d'}(\CC)^h$. Moreover, following the construction presented in \cite[Proposition 4.14]{HKM17}, shows that $d'=d$, so that $B$ is also a tuple of $d \times d$ matrices. Now, since $\cD_B$ is bounded, $\{I,B_1,\dots,B_{h}\}$ is linearly independent over $\RR$, hence $h \leq d^2-1$. On the other hand, the fact that $\cD_A^\circ \subset SM(\CC)^{d^2-1}$ is a projection of $\cD_B$ implies $h \geq d^2-1$. Thus, $h = d^2-1$. We conclude that $\cD_A^\circ = \cD_B$, which then implies that
 \[
\cD_A = (\cD_A^\circ)^\circ = \cD_B^\circ = \mconv(B).
 \]
 Arguing as before shows that $B$ is irreducible, so the above implies that, up to unitary equivalence, $B$ is the only free extreme point of $\cD_A$ and that $\cD_A$ is the matrix convex hull of its free extreme points. 
\end{proof}

\begin{remark}
 The above theorem gives an alternative proof strategy for Theorem \ref{thm:Paulidual}. In particular, it is sufficient to show that $P$ is a free extreme point of $\cD_P$. As previously discussed, this can be verified by solving a linear system of equations. 
\end{remark}

While Theorem \ref{theorem:DualFreeSpecs} gives a partial generalization of Theorem \ref{thm:Paulidual}, it is notable that if $d > 2$, then none of the free spectrahedra considered above can be self-dual. 

\begin{proposition}
    Suppose $d \geq 3$ and $d^2-d+2 \leq g \leq d^2-1$ and let $A \in SM_d(\CC)^{g}$. Also assume that $\cD_A$ is bounded. Then $\cD_A(1)$ is not self-dual. As a consequence, $\cD_A$ is not self-dual.
\end{proposition}  
\begin{proof}
     Since $\cD_A$ is bounded, the set $\{A_1,\dots,A_{g}\}$ is linearly independent over $\RR$, hence has real dimension $g$. On the other hand, the real dimension of $SM_d(\C)$ is $d^2$, and the dimension of the set of diagonal matrices in $SM_d (\CC)$ is $d$. Thus, the dimension of the real span of $\{A_1,\dots,A_{d^2-1}\}$ intersected with the set of diagonal matrices in $SM_d(\CC)$ is at least $2$. Additionally, since $\cD_A$ is bounded, the span of the $A_j$ cannot contain the identity matrix, so the dimension of the intersection in question cannot exceed $d-1$. Let $2 \leq k \leq d-1$ denote the dimension of this intersection. It follows that, up to an orthogonal linear transformation of the $A_j$ (see  (\ref{eq:ortho_example})), the matrices $A_1,\dots,A_{k}$ are linearly independent and diagonal. 
    
    Now, suppose towards a contradiction that $\cD_A(1)$ is self-dual. Then, even after the above orthogonal linear transformation, we must have that $\cD_A (1) = \overline{\mathbb{B}}_{d^2-1}$. It follows that $\cP_{k} \cD_A(1) = \overline{\mathbb{B}}_{k}$. On the other hand, since $\cD_A(1)$ is a Euclidean ball, 
    \[
    \cP_{k} \cD_A(1) = \{x \in \RR^{d-1} : (x_1,\dots,x_{k},0,\dots,0) \in \cD_A(1)\} = \cD_{A'} (1)
    \]
    where $A' = (A_1,\dots,A_{k})$. However, $\{A_1,\dots,A_{k}\}$ is a collection of linearly independent diagonal matrices and $\cD_{A'} (1)$ is bounded. It follows that $\cD_{A'} (1)$ is a polyhedron, so $\cD_{A'}(1) \neq \overline{\mathbb{B}}_{k}$, which is a contradiction. 
\end{proof}

\begin{remark}\label{rem:nobigdual}
    Taking $g=d^2-1$, the above proposition is intended to highlight that Theorem \ref{theorem:DualFreeSpecs} cannot be used to extend Theorem \ref{thm:Paulidual} in the natural way to $d^2-1$ variables when $d \geq 3$. This is not surprising, as it is known that the state space of $M_n$ for $n > 2$ is not a ball (see \cite[\S 1]{Avron} for a gentle introduction), in contrast with the Bloch ball of \cite{Bloch} when $n = 2$. Our result thus places more precise dimension bounds on a self-dual set. It remains an open question if there exists $A \in SM_d (\CC)^g$ such that $\cD_A$ is self-dual for some $d \geq 3$ and $g < d^2-d+2$.  
\end{remark}

\section*{Acknowledgments}
The views expressed in this document are those of the authors and do not reflect the official policy or position of the U.S. Naval Academy, the Department of the Navy, the Department of Defense, or the U.S. Government.

\vskip .1 in

We would also like to thank Michael Hartz for helpful comments.

 \Addresses

\end{document}